\documentclass{rmmcart}

\usepackage{amsmath,amsfonts,array,booktabs}

\title{An Automorphic Classification of Real Cubic Curves}

\author{Mark Bly}
\address{Department of Mathematics and Statistics, Coastal Carolina University, Conway, SC 29528}
\email{mbly@coastal.edu}
\thanks{Special thanks to S.B. Mulay for his patient guidance.}


\begin{document}
\begin{abstract}
The action of ring automorphisms of $\mathbb{R}[x,y]$ on real plane curves is considered. The orbits containing degree-three polynomials are computed, with one representative per orbit being selected.
\end{abstract}
\maketitle

Cubic curve classification has a deep history, dating back to Newton \cite{New}, Pl{\"u}cker \cite{Plu}, and Cayley \cite{Cay}. Many have addressed this problem by studying the action of ring automorphisms of $\mathbb{R}[x,y]$ that are linear in $x$ and $y$ on degree-three polynomials of $\mathbb{R}[x,y]\,$ \cite{Bur,For,Kor,Wein}. One such example was published in this journal \cite{Wein}.

This article serves two primary purposes. First, we wish to complete the classification that was started in \cite{Wein}. Second, we wish to extend this result to a classification that considers the action of all ring automorphisms of $\mathbb{R}[x,y]\,$.

Throughout, we will refer to the ring automorphisms of real-valued polynomials in two variables by $\text{Aut} \, \mathbb{R}[x,y]\,$. When explicitly defining a particular $\varphi \in \text{Aut} \, \mathbb{R}[x,y]\,$, we will frequently express $\varphi$ in the form $\big\langle \, p \, , \, q \, \big\rangle\,$, where $p=\varphi(x)$ and $q=\varphi(y)\,$. The group of automorphisms for which $p$ and $q$ are both linear (ie. the affine group) we will refer to by $\Gamma_2(\mathbb{R})\,$. We will also use $\sigma$ to refer to the signum function that maps zero to one. 

\section{Affine Classification}

We will now begin our affine classification by defining the relevant equivalence relation on $\mathbb{R}[x,y]\,$.

\begin{definition}
If $f,g$ are polynomials in $\mathbb{R}[x,y]\,$, then we say $f \sim g$ if there exists some $\theta \in \Gamma_2(\mathbb{R})$ and some $c \in \mathbb{R}^{\times}$ such that $\theta(f) \, = \, cg\,$.
\end{definition}

With respect to our equivalence relation $\sim\,$, we seek a list of polynomials comprised of exactly one representative from each equivalence class that contains a degree-three polynomial. As in \cite{Wein}, we will assume that the homogeneous degree-three component of our polynomials be in one of four canonical forms: $x^3+xy^2\,$, $x^3-xy^2\,$, $x^2y\,$, $x^3$.

\begin{proposition}\label{prop_x^3+xy^2}
If $f$ is a polynomial in $\mathbb{R}[x,y]$ of the form \[ x^3+xy^2+Ex^2+Fxy+Gy^2+\lambda(f) \] where $E,F,G$ are real numbers and $\lambda(f)$ is linear, then $f \sim g$ for some $g$ listed in Table \ref{affine}.
\end{proposition}

\begin{proof}
Let $\theta$ be as defined in Table \ref{x^3+xy^2_high} and consider $\theta(f)\,$. It follows that $\theta(f)$ is in one of the following forms, where $c,H,I,J$ are in $\mathbb{R}\,$.
\begin{align}
c\left(x^3+xy^2+x^2+Hx+Iy+J\right) \, , \\
c\left(x^3+xy^2+Hx+Iy+J\right) \, . 
\end{align}
Should $\theta(f)$ be of form (1), let $\theta'$ be $\big\langle \, x \, , \, \sigma(I)y \, \big\rangle\,$. Should $\theta(f)$ be of form (2), let $\theta'$ be as defined in Table \ref{x^3+xy^2_low}. Consider $\left(\theta'\circ\theta \right)(f)$ and the result follows.
\end{proof}

\begin{table}
\def\arraystretch{1.5}
\begin{tabular}{| >{\centering}m{1.15in} | >{\centering}m{2.75in} |}
\hline
$E \neq 3G$ & $\big\langle \, (E-3G)x-G \, , \, (E-3G)y-\frac{F}{2} \, \big\rangle$
\tabularnewline \hline
$E=3G$ & $\big\langle \, x-G \, , \, y-\frac{F}{2} \, \big\rangle$
\tabularnewline \hline
\end{tabular}
\newline
\caption{For polynomials $x^3+xy^2+Ex^2+Fxy+Gy^2 + \lambda(f)$}\label{x^3+xy^2_high}
\end{table}

\begin{table}
\def\arraystretch{1.5}
\begin{tabular}{| >{\centering}m{1.15in} | >{\centering}m{2.75in} |}
\hline
$I \neq 0$ & $\big\langle \, \sigma(J)|I|^{1/2}x \, , \, \sigma(J)\sigma(I)|I|^{1/2}y \, \big\rangle$
\tabularnewline \hline
$I=0\,$, $H \neq 0$ & $\big\langle \, \sigma(J)|H|^{1/2}x \, , \, \sigma(J)|H|^{1/2}y \, \big\rangle$
\tabularnewline \hline
$I=H=0\,$, $J\neq0$ & $\big\langle \, J^{1/3}x \, , \, J^{1/3}y \, \big\rangle$
\tabularnewline \hline
$I=H=J=0$ & $\big\langle \, x \, , \, y \, \big\rangle$
\tabularnewline \hline
\end{tabular}
\newline
\caption{For polynomials $c(x^3+xy^2+Hx+Iy+J)$}\label{x^3+xy^2_low}
\end{table}

\begin{proposition}\label{prop_x^3-xy^2}
If $f$ is a polynomial in $\mathbb{R}[x,y]$ of the form \[ x^3-xy^2+Ex^2+Fxy+Gy^2+\lambda(f) \] where $E,F,G$ are real numbers and $\lambda(f)$ is linear, then $f \sim g$ for some $g$ listed in Table \ref{affine}.
\end{proposition}

\begin{proof}
Let $\theta$ be as defined in Table \ref{x^3-xy^2_high} and consider $\theta(f)\,$. It follows that $\theta(f)$ is in one of the following forms, where $c,H,I,J$ are in $\mathbb{R}\,$.
\begin{align}
c\left(x^3-xy^2+Hx+Iy+J\right) \, , \\
c\left(x^3-xy^2-y^2+Hx+Iy+J \right) \, .
\end{align}
We will consider four example cases, two where $\theta(f)$ is of form (3) and two where $\theta(f)$ is of form (4)$\,$.

Assume that $\theta(f)$ is of form (3) and $I \neq 0\, ,\, |H|>|I|\,$. (We will neglect $c\,$, which can be factored out.) Define $\theta_1$ to be $\big\langle \, |I|^{1/2} x \, , \, \sigma(-I)|I|^{1/2}y \, \big\rangle\,$, and note that\[ \left(\theta_1 \circ \theta\right)(f) \, = \, |I|^{3/2} \, \left( \, x^3-xy^2+\frac{H}{|I|}x-y+\frac{J}{|I|^{3/2}} \, \right) \, . \]Subsequently let $\alpha:=-\sqrt{\frac{|H|+|I|}{8\:\!|I|}}\,$, and define $\theta_2$ to be \[ \Big\langle \: \alpha\left( \, x+ \sigma(H)\:\! y\, \right) \: , \: -\alpha\:\! \sigma(-H) \left( \, 3x-\sigma(H) \:\! y \,\right) \: \Big\rangle \, . \]It follows that $\left(\theta_2 \circ \theta_1 \circ \theta\right)(f)$ equals \[ -8\alpha^3 \left( \, x^3-xy^2+ \left(\sigma(-H)\frac{|H|-3|I|}{|H|+|I|}\right)x - y + \frac{J}{-8\alpha^3 |I|^{3/2}} \, \right) \, . \] It remains to consider the absolute value of $\frac{|H|-3|I|}{|H|+|I|}\,$. If $|H|-3|I|$ is positive, then $|H|-3|I|$ must have a smaller absolute value than $|H|+|I|\,$. Should $|H|-3|I|$ be negative, our assumption $|H|>|I|$ implies that \[ \frac{|H|-3|I|}{|H|+|I|} \, > \, \frac{-2\:\!|I|}{2\:\!|I|} \, = \, -1 \, . \] Regardless of case, we have that $x$-coefficient of $\frac{\left(\theta_2 \circ \theta_1 \circ \theta\right)(f)}{-8\alpha^3}$ must be less than $1$ in absolute value. Letting $\theta_3$ be defined as $\big\langle \, \sigma(J) x \, , \, \sigma(J) y \, \big\rangle\,$, note that the appropriate $\theta'$ in Table \ref{x^3-xy^2_low1} is exactly $\theta_3 \circ \theta_2 \circ \theta_1\,$, and the desired result follows from considering $\left(\theta' \circ \theta\right)(f)\,$.

Assume that $\theta(f)$ is of form (3) and $I=0\, ,\, H \neq 0\,$. (We will neglect $c\,$, which can be factored out.) Let $\beta$ be $-\sqrt{\frac{|H|}{8}}\,$, and define $\theta_1$ to be\[ \big\langle \, {-\beta\left( \, x+\sigma(H)\:\! y \, \right)} \: , \: {-\beta\left( \, 3x-\sigma(H)\:\! y \, \right)} \, \big\rangle \, . \]It follows that $\left(\theta_1 \circ \theta\right)(f)$ is equal to \[ 8\beta^3 \left( \, x^3-xy^2+\frac{H}{-8\beta^2}x+ \frac{|H|}{-8\beta^2}y+ \frac{J}{8\beta^3} \, \right) \, . \]Note that the $x$- and $y$-coefficient of $\frac{\left(\theta_1 \circ \theta\right)(f)}{8\beta^3}$ must be $\pm 1$ and ${-1}$, respectively. Letting $\theta_2$ be defined as $\big\langle \, \sigma(-J) x \, , \, \sigma(-J) y \, \big\rangle\,$, note that the appropriate $\theta'$ in Table \ref{x^3-xy^2_low1} is exactly $\theta_2 \circ \theta_1\,$, and the desired result follows from considering $\left(\theta' \circ \theta\right)(f)\,$.

Assume that $\theta(f)$ is of form $(4)$ and $\left|H+\frac{3}{4}\right| < |I|\,$. (We will neglect $c\,$, which can be factored out.) Define $\theta_1$ to be \[ \left\langle \, {-\frac{1}{2}}(x+y)-\frac{3}{4} \: , \: \sigma(I) \left({-\frac{1}{2}}(3x-y)-\frac{3}{4}\right) \, \right\rangle\, . \] It follows that $(\theta_1 \circ \theta)(f)$ is equal to $x^3-xy^2-y^2+H'x+I'y+J'\,$, where \[ H' \, = \, {-\frac{9}{8}}-\frac{H}{2}-\frac{3|I|}{2} \, ; \qquad I' \, = \, {-\frac{3}{8}}-\frac{H}{2}+\frac{|I|}{2}\, . \] First, observe that our assumed relation on $H$ and $I$ implies that $H \leq {-\frac{3}{4}}+|I|\,$, which subsequently yields \[ I' \, \geq \, {-\frac{3}{8}}-\left({-\frac{3}{8}}+\frac{|I|}{2}\right)+\frac{|I|}{2} \, = \, 0 \, . \] Second, observe that our assumed relation on $H$ and $I$ implies that $H+|I| \geq {-\frac{3}{4}}\,$, which subsequently yields \[ H'+I' \, = \, {-\frac{3}{2}}-(H+|I|) \, \leq \, {-\frac{3}{4}}\, . \] As such, $(\theta_1 \circ \theta)(f)$ is of the form of a polynomial from Table \ref{affine}$\,$.

Assume that $\theta(f)$ is of the form of $(4)$ and $H+\frac{3}{4}\geq|I|\,$. (We will neglect $c\,$, which can be factored out.) Define $\theta_1$ to be \[ \left\langle \, {-\frac{1}{2}}(x-y)-\frac{3}{4} \: , \: \sigma(I) \left({-\frac{1}{2}}(3x+y)-\frac{3}{4}\right) \, \right\rangle\, . \] It follows that $(\theta_1 \circ \theta)(f)$ is equal to $x^3-xy^2-y^2+H'x+I'y+J'\,$, where \[ H' \, = \, {-\frac{9}{8}}-\frac{H}{2}-\frac{3|I|}{2} \, ; \qquad I' \, = \, \frac{3}{8}+\frac{H}{2}-\frac{|I|}{2}\, . \] First, observe that our assumed relation on $H$ and $I$ implies that $H-|I| \geq {-\frac{3}{4}}\,$, which subsequently yields \[ I' \, \geq \, {\frac{3}{8}}+\left({-\frac{3}{8}}\right) \, = \, 0 \, . \] Second, observe that \[ H'+I' \, = \, {-\frac{3}{4}}-2|I| \, \leq \, {-\frac{3}{4}}\, . \] As such, $(\theta_1 \circ \theta)(f)$ is of the form of a polynomial from Table \ref{affine}$\,$.

In the remaining cases where $\theta(f)$ is of form (3), let $\theta'$ be as defined in Table \ref{x^3-xy^2_low1}. In the remaining cases where $\theta(f)$ is of form (4), let $\theta'$ be as defined in Table \ref{x^3-xy^2_low2}. Consider $\left(\theta'\circ\theta \right)(f)$ and the result follows.
\end{proof}

\begin{table}
\def\arraystretch{1.5}
\begin{tabular}{| >{\centering}m{1.15in} | >{\centering}m{2.75in} |}
\hline
$E \neq -3G$ & $\big\langle \, (-\frac{E}{3}-G)x-\frac{E}{3} \, , \, (-\frac{E}{3}-G)y+\frac{F}{2} \, \big\rangle$
\tabularnewline \hline
$E=-3G$ & $\big\langle \, x-\frac{E}{3} \, , \, y+\frac{F}{2} \, \big\rangle$
\tabularnewline \hline
\end{tabular}
\newline
\caption{For polynomials $x^3-xy^2+Ex^2+Fxy+Gy^2 + \lambda(f)$}\label{x^3-xy^2_high}
\end{table}

\begin{table}
\def\arraystretch{1.5}
\begin{tabular}{| >{\centering}m{1.15in} | >{\centering}m{2.75in} |}
\hline
$I \neq 0\,$, $|H|\leq|I|$ & $\big\langle \, \sigma(J)|I|^{1/2}x \, , \, \sigma(J)\sigma(-I)|I|^{1/2}y \, \big\rangle$
\tabularnewline \hline
$I \neq 0\,$, $|H|>|I|$ & $\big\langle \, \alpha |I|^{1/2}\sigma(J)\left(x+\sigma(H)y\right) \, , \qquad\qquad\qquad\qquad\!$
\tabularnewline
 &  $\qquad\qquad \, -\alpha |I|^{1/2} \sigma(J) \sigma(IH) \left(3x-\sigma(H)y\right) \, \big\rangle$
\tabularnewline \hline
$I=0\,$, $H\neq0$ & $\big\langle \, \beta \, \sigma(J) \left(x+\sigma(H)y\right) \, , \, \beta\, \sigma(J) \left(3x-\sigma(H)y\right) \, \big\rangle$
\tabularnewline \hline
$I=H=0\,$, $J\neq0$ & $\big\langle \, J^{1/3}x \, , \, J^{1/3}y \, \big\rangle$
\tabularnewline \hline
$I=H=J=0$ & $\big\langle \, x \, , \, y \, \big\rangle$
\tabularnewline \hline
\end{tabular}
\newline
\caption{For polynomials $c(x^3-xy^2+Hx+Iy+J)$}\label{x^3-xy^2_low1}
\end{table}

\begin{table}
\def\arraystretch{1.5}
\begin{tabular}{| >{\centering}m{1.15in} | >{\centering}m{2.75in} |}
\hline
$\left|H+\frac{3}{4}\right| < |I|$ & $\big\langle \, -\frac{1}{2}(x+y)-\frac{3}{4} \, , \, \sigma(I)\left(-\frac{1}{2}(3x-y)-\frac{3}{4}\right) \, \big\rangle$
\tabularnewline \hline
$H+ \frac{3}{4} \geq |I|$ & $\big\langle \, -\frac{1}{2}(x-y)-\frac{3}{4} \, , \, \sigma(I)\left(-\frac{1}{2}(3x+y)-\frac{3}{4}\right) \, \big\rangle$
\tabularnewline \hline
$H+ \frac{3}{4} \leq -|I|$ & $\big\langle \, x \, , \, \sigma(I)y \, \big\rangle$
\tabularnewline \hline
\end{tabular}
\newline
\caption{For polynomials $c(x^3-xy^2-y^2+Hx+Iy+J)$}\label{x^3-xy^2_low2}
\end{table}

\begin{proposition}\label{prop_x^2y}
If $f$ is a polynomial in $\mathbb{R}[x,y]$ of the form \[ x^2y+Ex^2+Fxy+Gy^2+\lambda(f) \] where $E,F,G$ are real numbers and $\lambda(f)$ is linear, then $f \sim g$ for some $g$ listed in Table \ref{affine}.
\end{proposition}

\begin{proof}
Let $\theta$ be as defined in Table \ref{x^2y_high} and consider $\theta(f)\,$. It follows that $\theta(f)$ is in one of the following forms, where $c,H,I,J$ are in $\mathbb{R}\,$.
\begin{align}
c(x^2y+y^2+Hx+Iy+J) \, , \\ 
c(x^2y+Hx+Iy+J)\, .\
\end{align}
Should $\theta(f)$ be of form (5), let $\theta'$ be defined as in Table \ref{x^2y_low1}. Should $\theta(f)$ be of form (6), let $\theta'$ be defined as in Table \ref{x^2y_low2}.  Consider $\left(\theta'\circ\theta \right)(f)$ and the result follows.
\end{proof}

\begin{table}
\def\arraystretch{1.5}
\begin{tabular}{| >{\centering}m{1.15in} | >{\centering}m{2.75in} |}
\hline
$G \neq 0$ & $\big\langle \, x - \frac{F}{2} \, , \, \frac{1}{G}y-E \, \big\rangle$
\tabularnewline \hline
$G=0$ & $\big\langle \, x-\frac{F}{2} \, , \, y-E \, \big\rangle$
\tabularnewline \hline
\end{tabular}
\newline
\caption{For polynomials $x^2y+Ex^2+Fxy+Gy^2 + \lambda(f)$}\label{x^2y_high}
\end{table}

\begin{table}
\def\arraystretch{1.5}
\begin{tabular}{| >{\centering}m{1.15in} | >{\centering}m{2.75in} |}
\hline
$H \neq 0$ & $\big\langle \, {-H}^{1/3}x \, , \, H^{2/3}y \, \big\rangle$
\tabularnewline \hline
$H=0\,$, $I \neq 0$ & $\big\langle \, |I|^{1/2}x \, , \, |I|y \, \big\rangle$
\tabularnewline \hline
$H=I=0\,$, $J\neq0$ & $\big\langle \, |J|^{1/4}x \, , \, |J|^{1/2}y \, \big\rangle$
\tabularnewline \hline
$H=I=J=0$ & $\big\langle \, x \, , \, y \, \big\rangle$
\tabularnewline \hline
\end{tabular}
\newline
\caption{For polynomials $c(x^2y+y^2+Hx+Iy+J)$}\label{x^2y_low1}
\end{table}

\begin{table}
\def\arraystretch{1.5}
\begin{tabular}{| >{\centering}m{1.15in} | >{\centering}m{2.75in} |}
\hline
$HI \neq 0$ & $\big\langle \, \sigma(-HJ)|I|^{1/2}x \, , \, \sigma(-HJ) \frac{-H}{|I|^{1/2}}y \, \big\rangle$
\tabularnewline[1.5pt] \hline
$HJ \neq 0\,$, $I = 0$ & $\big\langle \, -\frac{J}{H}x \, , \, \frac{H^2}{J}y \, \big\rangle$
\tabularnewline \hline
$H\neq0\,$, $I=J=0$ & $\big\langle \, x \, , \, -Hy \, \big\rangle$
\tabularnewline \hline
$H=0\,$, $IJ\neq0$ & $\big\langle \, |I|^{1/2}x \, , \, \frac{J}{|I|}y \, \big\rangle$
\tabularnewline \hline
$H=J=0\,$, $I\neq0$ & $\big\langle \, |I|^{1/2}x \, , \, y \, \big\rangle$
\tabularnewline \hline
$H=I=0\,$, $J\neq0$ & $\big\langle \, x \, , \, -Jy \, \big\rangle$
\tabularnewline \hline
$H=I=J=0$ & $\big\langle \, x \, , \, y \, \big\rangle$
\tabularnewline \hline
\end{tabular}
\newline
\caption{For polynomials $c(x^2y+Hx+Iy+J)$}\label{x^2y_low2}
\end{table}

\begin{proposition}\label{prop_x^3}
If $f$ is a polynomial in $\mathbb{R}[x,y]$ of the form \[ x^3+Ex^2+Fxy+Gy^2+\lambda(f) \] where $E,F,G$ are real numbers and $\lambda(f)$ is linear, then $f \sim g$ for some $g$ listed in Table \ref{affine}.
\end{proposition}

\begin{proof}
Let $\theta$ be defined as in Table \ref{x^3_high} and consider $\theta(f)\,$.

We will consider an example case where $G \neq 0\,$. Noting Table \ref{x^3_high}, we see that $\theta$ is of the form $\big\langle \, ax+r \, , \, cx+dy \, \big\rangle$ where $a=\sigma(-G)\,$, $r=\frac{F^2-4EG}{12G}\,$, $c=\frac{F}{2\:\!|G|}\,$, and $d=\frac{1}{\sqrt{|G|}}\,$. Subsequently note the terms of $\theta(f)$ of degree greater than one are equal to
\[ \left(a^3\right)x^3 + \left(3a^2r+Ea^2+Fac+Gc^2\right)x^2+\left(Fad+2Gcd\right)xy+\left(Gd^2\right)y^2 \, . \]Substituting appropriately for $a,b,c,d$ yields that the $x^2$- and $xy$-coefficients in $\theta(f)$ are \[ \frac{F^2-4EG}{4G} + E - \frac{F^2}{2G} + \frac{F^2}{4G} \, = \, 0 \, , \qquad \left(\, F\:\!\sigma(-G) + F \:\!\sigma(G)\, \right) d \, = \, 0 \, , \] respectively. Moreover, the $x^3$- and $y^2$-coefficients in $\theta(f)$ are $\sigma(-G)$ and $\sigma(G)\,$, respectively. Hence, $\theta(f)$ is of the form of (7) below.

From inspection of the remaining cases, it follows that $\theta(f)$ is in one of the following forms, where $c,H,I,J$ are in $\mathbb{R}\,$.
\begin{align}
c\left( x^3-y^2+Hx+Iy+J \right) \, , \\
c\left( x^3-xy+Hx+Iy+J \right) \, , \\
c\left( x^3+Hx+Iy+J \right) \, .
\end{align}
Should $\theta(f)$ be of form (7), let $\theta'$ be defined as in Table \ref{x^3_low1}. Should $\theta(f)$ be of form (8), let $\gamma := I^3+IH+J$ and $\theta'$ be defined as in Table \ref{x^3_low2}. Should $\theta(f)$ be of form (9), let $\theta'$ be defined as in Table \ref{x^3_low3}. Consider $\left(\theta'\circ\theta\right)(f)$ and the result follows.
\end{proof}

\begin{table}
\def\arraystretch{1.5}
\begin{tabular}{| >{\centering}m{1.15in} | >{\centering}m{2.75in} |}
\hline
$G \neq 0$ & $\big\langle \, \sigma(-G)x + \frac{F^2-4EG}{12G} \, , \, \frac{F}{2|G|}x+\frac{1}{\sqrt{|G|}}y \, \big\rangle$
\tabularnewline \hline
$G=0\,$, $F \neq 0$ & $\big\langle \, x-\frac{E}{3} \, , \, -\frac{1}{F}y \, \big\rangle$
\tabularnewline \hline
$G=F=0$ & $\big\langle \, x- \frac{E}{3} \, , \, y \, \big\rangle$
\tabularnewline \hline
\end{tabular}
\newline
\caption{For polynomials $x^3+Ex^2+Fxy+Gy^2 + \lambda(f)$}\label{x^3_high}
\end{table}

\begin{table}
\def\arraystretch{1.5}
\begin{tabular}{| >{\centering}m{1.15in} | >{\centering}m{2.75in} |}
\hline
$|H| \neq 0$ & $\big\langle \, |H|^{1/2}x \, , \, |H|^{3/4}y+\frac{I}{2} \, \big\rangle$
\tabularnewline \hline
$H=0\,$, $\frac{I^2}{4}+J \neq 0$ & $\big\langle \, \big|\frac{I^2}{4}+J\big|^{1/3}x \, , \, \big|\frac{I^2}{4}+J\big|^{1/2}y + \frac{I}{2} \, \big\rangle$
\tabularnewline \hline
$H=\frac{I^2}{4}+J=0$ & $\big\langle \, x \, , \, y+\frac{I}{2} \, \big\rangle$
\tabularnewline \hline
\end{tabular}
\newline
\caption{For polynomials $c(x^3-y^2+Hx+Iy + J)$}\label{x^3_low1}
\end{table}

\begin{table}
\def\arraystretch{1.5}
\begin{tabular}{| >{\centering}m{1.15in} | >{\centering}m{2.75in} |}
\hline
$\gamma \neq 0$ & $\big\langle \, \gamma^{1/3}x +I \, , \, 3\gamma^{1/3}Ix+\gamma^{2/3}y+3I^2+H \, \big\rangle$
\tabularnewline \hline
$\gamma = 0$ & $\big\langle \, x+I \, , \, 3Ix+y+3I^2+H \, \big\rangle$
\tabularnewline \hline
\end{tabular}
\newline
\caption{For polynomials $c(x^3-xy+Hx+Iy+J)$}\label{x^3_low2}
\end{table}

\begin{table}
\def\arraystretch{1.5}
\begin{tabular}{| >{\centering}m{1.15in} | >{\centering}m{2.75in} |}
\hline
$I \neq 0$ & $\big\langle \, x \, , \, -\frac{H}{I}x-\frac{1}{I}y-\frac{J}{I} \, \big\rangle$
\tabularnewline \hline
$I=0\,$, $H \neq 0$ & $\big\langle \, \sigma(J)|H|^{1/2}x \, , \, y \, \big\rangle$
\tabularnewline \hline
$I=H=0\,$, $J \neq 0$ & $\big\langle \, J^{1/3}x \, , \, y \, \big\rangle$
\tabularnewline \hline
$I=H=J=0$ & $\big\langle \, x \, , \, y \, \big\rangle$
\tabularnewline \hline
\end{tabular}
\newline
\caption{For polynomials $c(x^3+Hx+Iy + J)$}\label{x^3_low3}
\end{table}

To conclude that the list of polynomials in Table \ref{affine} contains only one representative from each equivalence class of degree-three polynomials in $\mathbb{R}[x,y]\,$, it remains to show that the polynomials listed are pairwise inequivalent with respect to $\sim\,$. Observing that the canonical forms $x^3+xy^2\,$, $x^3-xy^2\,$, $x^2y\,$, and $x^3$ are pairwise inequivalent with respect to $\sim\,$, we can proceed by inspecting each canonical form individually.

\begin{proposition}\label{prop2_x^3+xy^2}
Assume $f,g$ are polynomials listed in Table \ref{affine} with canonical form $x^3+xy^2\,$. If $f \sim g\,$, then $f=g\,$.
\end{proposition}

\begin{proof}
Let $\theta \in \Gamma_2(\mathbb{R})$ be such that $\theta(f)=cg$ for some $c \in \mathbb{R}^{\times}\,$, and express $\theta$ as $\big\langle \, Ax+By+R \, , \, Cx+Dy+S \, \big\rangle\,$ where $A,B,R,C,D,S$ are in $\mathbb{R}\,$. Given that $x^3+xy^2$ factors as $x(x^2+y^2)\,$, it follows that $B$ and $C$ must both be zero. Since the $xy$- and $y^2$-coefficients of $\theta(f)$ must be zero, it follows that $R=S=0\,$. As such, $\theta$ is of the form $\big\langle \, Ax \, , \, Dy \, \big\rangle\,$. The desired result follows from inspection.
\end{proof}

\begin{proposition}\label{prop2_x^3-xy^2}
Assume $f,g$ are polynomials listed in Table \ref{affine} with canonical form $x^3-xy^2\,$. If $f \sim g\,$, then $f=g\,$.
\end{proposition}

\begin{proof}
Let $\theta \in \Gamma_2(\mathbb{R})$ be such that $\theta(f)=cg$ for some $c \in \mathbb{R}^{\times}\,$, and write $\theta$ in the form $\big\langle \, Ax+By+R \, , \, Cx+Dy+S \, \big\rangle\,$ where $A,B,R,C,D,S$ are in $\mathbb{R}\,$. Since this canonical form factors as $x(x+y)(x-y)\,$, it follows that $\theta(x)$ is of the form $Ax+R\,$, $A(x+y)+R\,$, or $A(x-y)+R\,$. It is a straightforward computation to determine that $\theta(y)$ must (respectively) be of the form $\pm Ay+S\,$, $\pm A(3x-y)+S\,$, or $\pm A(3x+y)+S$. Should $\theta$ be of the form $\big\langle \, Ax+R \, , \, \pm Ay+S \, \big\rangle\,$, observe that $R$ and $S$ must be zero since the $x^2$- and $xy$-coefficients of $\theta(f)$ must be zero. As such, the desired result follows from inspection. Observing that $\theta(f)=cg$ implies $f(1,y)$ and $g(1,y)$ have equal $y^2$-coefficients will assist with the remaining cases. We will consider each individually.

Assume that $f(1,y)$ has a $y^2$-coefficient of ${-1}\,$. Should $\theta$ be of the form $\big\langle \, A(x+y)+R \, , \, \pm A(3x-y)+S \, \big\rangle\,$, observe that the $xy$- and $x^2$-coefficients of $\theta(f)$ being zero implies that $S=R=0\,$. This implies our desired result for $x^3-xy^2+1$ and $x^3-xy^2\,$. For the remaining polynomials to consider, since the $x^3$- and $y$-coefficients of $\theta(f)$ must be opposites, it follows that $-8A^3=-HA \mp A\,$. Since the $x^3$-coefficient of $\theta(f)$ must be at least as large as the $x$-coefficient of $\theta(f)$ in absolute value, we also have that $|\pm 3A+HA| \leq |-8A^3|\,$. Combining these observations yields $|H \pm 3| \leq |H \pm 1|\,$. Given that $H$ must be contained on the interval $[{-1},1]\,$, it is forced that $H=\pm 1$ and $A=\pm\frac{1}{2}\,$. As such, $\theta(f) \, = \, \pm f$ and we achieve our desired result. The case where $\theta(x)$ is of the form $A(x-y)+R$ is similar.

Assume that $f(1,y)$ has a $y^2$-coefficient of ${-2}\,$. Should $\theta$ be of the form $\big\langle \, A(x+y)+R \, , \, \pm A(3x-y)+S \, \big\rangle \,$, inspection of the $x^2$-, $xy$-, and $y^2$-coefficients yields this system of equations. (Please note that within this proof, any use of the notation $\pm$ or $\mp$ will be used such that the top symbol will correspond to the case $\theta(y)=A(3x-y)+S$ and the bottom symbol will correspond to the case $\theta(y)=-A(3x-y)+S\,$.) 
\begin{align*}
{-6 R} \mp 6 S - 9  \, &= \, 0 \, ,  \\
12 R \mp 4 S + 6  \, &= \, 0 \, , \\
2R \pm 2S - 1 \, &= \, 8A \, .
\end{align*}
Rearranging these equations and solving yields \[ \left[ \begin{array}{c} R \\ S \\ A \end{array} \right] \, =\, \left[ \begin{array}{ccc} -6 & \mp 6 & 0 \\ 12 & \mp 4 & 0 \\ 2 & \pm 2 & -8 \end{array}\right]^{-1} \left[ \begin{array}{c} 9 \\ {-6} \\ 1 \end{array} \right] \, = \, \left[ \begin{array}{c} {-3/4} \\ {\mp 3/4} \\ {-1/2} \end{array} \right]\, . \] For convenience, let us refer to the $x$- and $y$-coefficients of $\theta(f)$ as $H'$ and $I'$ respectively. With the values forced upon $R,S,A$ in the expression of $\theta\,$, it now follows that: $H'$ must be ${-\frac{9}{8}}-\frac{H}{2} \mp \frac{3I}{2}\,$; $I'$ must be ${-\frac{3}{8}}-\frac{H}{2}\pm \frac{I}{2}\,$; and $c$ is $1$ (ie. $\theta(f)=g\,$). Given the restriction on the $x$- and $y$-coefficients of $g$ from Table \ref{affine}$\,$, it further follows that \[ H'+I' \, = \, {-\frac{3}{2}}-H \mp I \, \leq \, {-\frac{3}{4}} \, . \] In the case that $\theta(y)$ is equal to $-\frac{1}{2}(3x-y)-\frac{3}{4}\,$, we get the result $-H-I \leq \frac{3}{4}\,$. Given the restriction on $H+I$ from Table \ref{affine}$\,$, it must be that $H+I$ is equal to $\frac{3}{4}\,$. As such, inspection on the expressions for $I'$ and $H'$ above yields \[ I'\,=\, {-\frac{3}{8}}-\frac{H}{2}-\frac{I}{2}+I \, = \, I\, ; \qquad  H'={-\frac{3}{4}}-I' \, . \] These equalities imply that $(H,I)$ must equal $(H',I')\,$. It follows that $\theta(f)$ must also equal $f\,$, and the desired result is achieved in this case. In the case that $\theta(y)$ is equal to $\frac{1}{2}(3x-y)+\frac{3}{4}\,$, we get the result $-H+I \leq \frac{3}{4}\,$. Subtracting $2I$ and multiplying by $-1$ yields $H+I \geq {-\frac{3}{4}}+2I\,$. Since $I$ is nonnegative, the restriction on $H+I$ from Table \ref{affine} implies that $I$ must be zero and subsequently $H$ must equal $-\frac{3}{4}\,$. Substituting accordingly into the expressions for $H'$ and $I'$ above, we have that $(H',I')$ must equal $(H,I)\,$. It follows that $\theta(f)$ must also $f\,$, and the desired result is achieved in this case as well. Should $\theta(x)$ be of the form $A(x-y)\,$, the result is similar. 
\end{proof}

\begin{proposition}\label{prop2_x^2y}
Assume $f,g$ are polynomials listed in Table \ref{affine} with canonical form $x^2y\,$. If $f \sim g\,$, then $f=g\,$.
\end{proposition}

\begin{proof}
The proof is similar to the that of Proposition \ref{prop2_x^3+xy^2} but with inspection of $xy$- and $x^2$-coefficients to yield that $R=S=0\,$.
\end{proof}

\begin{proposition}\label{prop2_x^3}
Assume $f,g$ are polynomials listed in Table \ref{affine} with canonical form $x^3\,$. If $f \sim g\,$, then $f=g\,$.
\end{proposition}

\begin{proof}
Let $\theta \in \Gamma_2(\mathbb{R})$ be such that $\theta(f)=cg$ for some $c \in \mathbb{R}^{\times}\,$, and write $\theta$ in the form $\big\langle \, Ax+By+R \, , \, Cx+Dy+S \, \big\rangle\,$ where $A,B,R,C,D,S$ are in $\mathbb{R}\,$. Given the canonical form $x^3\,$, it follows that $B$ is zero and $D$ is nonzero. As such, the polynomials $f(1,y)$ and $g(1,y)$ must have equal degree. We will consider three cases accordingly.

Should the degree of $f(1,y)$ be less than one, observe that the $x^2$-coefficient of $\theta(f)$ being zero implies that $R=0\,$. Hence, $\theta$ is of the form $\big\langle \, Ax \, , \, Cx+Dy+S \, \big\rangle$ and the desired result follows from inspection. Should the degree of $f(1,y)$ equal one, observe that $f$ is either $x^3-xy\,$, $x^3-xy+1\,$, or $x^3-y\,$. Noting that $B=0$ forces the $xy$-coefficient of $\theta (x^3-y)$ to be zero, it follows that $f$ equaling $x^3-y$ implies that $f=g\,$. Further, since $x^3-xy$ is reducible and $x^3-xy+1$ is irreducible, our desired result follows in this case from the fact that $\theta$ is a ring automorphism of $\mathbb{R}[x,y]\,$. Should the degree of $f(1,y)$ be greater than one, observe the the $xy$-coefficient of $\theta(f)$ being zero implies that $C=0\,$. Further inspection of the $x^2$- and $y$-coefficients of $\theta(f)$ implies that $R=S=0\,$. As such, $\theta$ must be the of the form $\big\langle \, Ax \, , \, Dy \, \big\rangle$ and the result follows by inspection.
\end{proof}

We have now established an affine classification of cubic curves, specifically one that finishes the work started in \cite{Wein}.

\begin{theorem}\label{thm1}
Every degree-three polynomial in $\mathbb{R}[x,y]$ is equivalent with respect to $\sim$ to some polynomial in Table \ref{affine}. Moreover, the polynomials in Table \ref{affine} are pairwise inequivalent with respect to $\sim\,$.
\end{theorem}

\begin{proof}
The result follows from the propositions of this section.
\end{proof}

\begin{table}
\def\arraystretch{1.15}
\begin{tabular}{| m{3.9in} |}
\hline
\vspace{4pt}
$x^3+xy^2+x^2+Hx+Iy+J \, ; \;\; H,J \in \mathbb{R}\, , \; I \geq 0$ \\
$x^3+xy^2+y+ Hx+J \, ; \;\; H \in \mathbb{R} \, , \; J \geq 0$ \\
$x^3+xy^2\pm x + J \, ; \;\; J \geq 0$ \\
$x^3+xy^2+1$ \\
$x^3+xy^2$ 
\vspace{4pt}
\tabularnewline \hline
\vspace{4pt}
$x^3-xy^2-y^2+Hx+Iy+J \, ; \;\; J \in \mathbb{R} \, , \; I\geq 0 \, , \; H+I \leq {-\frac{3}{4}}$ \\
$x^3-xy^2-y+Hx + J \, ; \;\; J \geq 0 \, , \;  H \in [{-1},1]$ \\
$x^3-xy^2+1$ \\
$x^3-xy^2$
\vspace{4pt}
\tabularnewline \hline
\vspace{4pt}
$x^2y+y^2-x+Iy+J \, ; \;\; I,J \in \mathbb{R}$ \\
$x^2y +y^2 \pm y + J \, ; \;\; J \in \mathbb{R}$ \\
$x^2y+y^2 \pm 1$ \\
$x^2y+y^2$ \\
$x^2y - x \pm y + J \, ; \;\; J \geq 0$ \\
$x^2y \pm y+1$ \\
$x^2y \pm y$ \\
$x^2y-x+1$ \\
$x^2y-x$ \\
$x^2y-1$ \\
$x^2y$
\vspace{4pt}
\tabularnewline \hline
\vspace{4pt}
$x^3-y^2 \pm x + J \, ; \;\; J \in \mathbb{R}$ \\
$x^3-y^2 \pm 1$ \\
$x^3-y^2$ \\
$x^3-y$ \\
$x^3-xy+1$ \\
$x^3-xy$ \\
$x^3 \pm x + J \,; \;\; J\geq 0$ \\
$x^3+1$ \\
$x^3$
\vspace{4pt}
\tabularnewline \hline
\end{tabular}
\newline
\caption{Affine/Automorphic Classification of Cubic Curves}\label{affine}
\end{table}

\section{Automorphic Classification}

We will use our recently established affine classification to develop an automorphic classification. This approach follows naturally given that the equivalence relation we define below is coarser than $\sim\,$.

\begin{definition}
If $f,g$ are polynomials in $\mathbb{R}[x,y]\,$, then we say $f \approx g$ if there exists some $\varphi \in \text{Aut}\,\mathbb{R}[x,y]$ and some $c \in \mathbb{R}^{\times}$ such that $\varphi(f) \, = \, cg\,$. More generally, if $K$ is a field and $f,g$ are polynomials in $K[x,y]\,$, then we say $f \approx_K g$ if there exists some $\phi \in \text{Aut}\,K[x,y]$ that fixes $K$ and some $c \in K^{\times}$ such that $\phi(f) \, = \, cg\,$.
\end{definition}

To allude to the additional complexity associated with the equivalence relation $\approx$ in comparison to $\sim\,$, let us consider an example involving $f=x^3-y^2$ and $g=-x^3+y^2$. If we suppose that there exists a $\theta$ in $\Gamma_2(\mathbb{R})$ such that $\theta(f)=g\,$, we can utilize the fact that $\theta$ fixes the degree of any monomial term to imply contradiction on the forced relation of $\theta(-y^2)$ being to equal $y^2\,$ (as there is no real number whose square is $-1\,$). Instead, if we suppose that there exists a $\varphi$ in $\text{Aut}\,\mathbb{R}[x,y]$ such that $\varphi(f) = g\,$, we can make no such claim about the image of $-y^2\,$. In principle, it is feasible that for some positive integer $k\,$: $\varphi(x)$ be of degree $2k\,$; $\varphi(y)$ be of degree $3k\,$; and the expressions for $\varphi(x^3)$ and $\varphi(-y^2)$ add to one another to cancel  all terms of degree greater than three to yield $-x^3+y^2\,$. With circumstances as such, additional considerations regarding $f$ and $g$ will be required to arrive at a conclusion (of contradiction) for this argument. In fact, within Proposition \ref{auto_x^3} we will work through details of a proof regarding this precise situation.

Our general strategy within this section will be to consider the equivalence classes with respect to $\approx$ that contain at least one polynomial from Table \ref{affine}. We will develop a list of those polynomials that contain exactly one representative from each equivalence class. Observing that the canonical forms $x^3+xy^2\,$, $x^3-xy^2\,$, $x^2y\,$, and $x^3\,$ have distinct factorization structures in $\mathbb{R}[x,y]$ and are hence pairwise inequivalent with respect to $\approx\,$, we can proceed by investigating each form individually.

In the two subsequent propositions, the following definition and lemma will prove useful.

\begin{definition}
If $f$ is a polynomial in $\mathbb{R}[x,y]\,$, let
\begin{align*}
\text{AutDeg}(f) \, &:= \, \min \lbrace \, \deg(g) \, \mid \, g \approx f \, \rbrace \, .
\end{align*}
\end{definition}

\begin{lemma}\label{lemma_x^2+-y^2}
$\emph{AutDeg}(x^2 \pm y^2) \, = \, 2\,.$
\end{lemma} 

\begin{proof}
Observe that $x^2-y^2=0$ defines crossing lines in $\mathbb{R}^2$ and $x^2+y^2=0$ defines an isolated point in $\mathbb{R}^2\,$. Also observe that neither of these curves could be mapped to a line, an empty variety, or all of $\mathbb{R}^2$ by a map in $\text{Aut} \, \mathbb{R}[x,y]$ (ie. a bi-polynomial homeomorphism of $\mathbb{R}^2\,$). Hence, $\text{AutDeg}(x^2 \pm y^2)$ could not be less than two.
\end{proof}

\begin{proposition}\label{auto_x^3+xy^2}
Assume $f,g$ are polynomials listed in Table \ref{affine} with canonical form $x^3+xy^2\,$. If $f \approx g\,$, then $f = g\,$.
\end{proposition}

\begin{proof}
Let $\varphi \in \text{Aut}\,\mathbb{R}[x,y]$ be such that $\varphi(f)=cg$ for some $c \in \mathbb{R}^{\times}\,$, and express $\varphi$ in the form $\big\langle \, p \, , \, q \, \big\rangle\,$. We will consider three cases. Suppose $\deg(p)$ is greater than $\deg(q)\,$, and observe $\deg(g)$ would be determined by $\deg(p^3)\,$. This leads to contradiction since $\deg(p)$ must be at least two in this case. Supposing $\deg(p)$ is less than $\deg(q)$ leads to a similar contradiction when inspecting the value of $\deg(pq^2)\,$. It is left to consider when $\deg(p)$ and $\deg(q)$ are equal. Suppose that $p,q$ are not linear. From inspection of the polynomials in Table \ref{affine}, it follows that $p(p^2+q^2)$ must be of degree three or less. This demands $\deg(p)=2$ and $\deg(p^2+q^2)=1\,$. But, this is a contradiction since $\text{AutDeg}(x^2+y^2)$ equals two via Lemma \ref{lemma_x^2+-y^2}. Hence, $p,q$ must both be linear, and the result follows from Theorem \ref{thm1}.
\end{proof}

\begin{proposition}\label{auto_x^3-xy^2}
Assume $f,g$ are polynomials listed in Table \ref{affine} with canonical form $x^3-xy^2\,$. If $f \approx g\,$, then $f = g\,$.
\end{proposition}

\begin{proof}
Let $\varphi \in \text{Aut}\,\mathbb{R}[x,y]$ be such that $\varphi(f)=cg$ for some $c \in \mathbb{R}^{\times}\,$, and express $\varphi$ in the form $\big\langle \, p \, , \, q \, \big\rangle\,$. Should $f(1,y)$ have a $y^2$-coefficient of ${-1}\,$, the result follows similarly to Proposition \ref{auto_x^3+xy^2} using the fact that $\text{AutDeg}(x^2-y^2)$ equals two via Lemma \ref{lemma_x^2+-y^2}. Hence, it remains to consider when $f(1,y)$ has a $y^2$-coefficient of ${-2}\,$. We will proceed in three cases. Suppose $\deg(p)$ is greater than $\deg(q)\,$, and observe that $\deg(g)$ would be determined by $\deg(p^3)\,$. This leads to contradictions since $\deg(p)$ must be at least two in this case. Supposing $\deg(p)$ is less than $\deg(q)$ leads to a similar result upon inspecting the degree of $\deg(pq^2)\,$. It remains to consider the case when $\deg(p)$ and $\deg(q)$ are equal. Suppose $p,q$ are not linear and note that $\varphi(f)$ is of the form $p(p^2-q^2+H)+(-q^2+Iq+J)\,$. Since $\text{AutDeg}(x^2-y^2)$ is not less than two, the restriction on $\deg(g)$ forces the relation \[ \deg(p)+\deg(p^2-q^2+H) \, = \, 2\deg(q)\, . \]Since we are in the case where $\deg(p)$ equals $\deg(q)\,$, it follows that $\deg(p^2-q^2)$ is equal to $\deg(p)\,$. In particular, this implies that $\deg(p+q)$ and $\deg(p-q)$ could not both be equal to $\deg(p)\,$. But, observe that should either $\deg(p+q)$ or $\deg(p-q)$ be less than $\deg(p)\,$, then the conjugate factor must have degree of exactly $\deg(p)\,$. This implies that either $p+q$ or $p-q$ must be a constant, which violates the algebraic independence of $p$ and $q\,$. This contradiction implies that $p,q$ must be linear and the result follows from Theorem \ref{thm1}.
\end{proof}

\begin{proposition}\label{auto_x^2y}
Assume $f,g$ are polynomials listed in Table \ref{affine} with canonical form $x^2y\,$. If $f \approx g\,$, then $f = g\,$.
\end{proposition}

\begin{proof}
Let $\varphi \in \text{Aut}\,\mathbb{R}[x,y]$ be such that $\varphi(f)=cg$ for some $c \in \mathbb{R}^{\times}\,$, and express $\varphi$ in the form $\big\langle \, p \, , \, q \, \big\rangle\,$. We will consider three cases. First, supposing $\deg(q)$ is less than $\deg(p)$ implies that $\deg(g)$ must be more than three. This is a contradiction. Second, assume $\deg(q)$ equals $\deg(p)\,$. The fact that $\deg(g)$ equals three forces $p,q$ to be linear, and hence $\varphi$ is in $\Gamma_2(\mathbb{R})$ and our result follows from Theorem \ref{thm1}. Finally, assume $\deg(q)$ is greater than $\deg(p)\,$, and note from Table \ref{affine} that $\varphi(f)$ can be written as $q \, f_1 + f_2 \,$, where \[ f_1 \in \{ \, p^2+q+I \, ,\, p^2 \pm 1 \, , \, p^2 \,\; \vert \;\, I \in \mathbb{R} \, \} \, ; \;\, f_2 \in \{ \, -p+J \, , \, J \,\; \vert \;\, J \in \mathbb{R} \, \} \, . \] Since $p,q$ are algebraically independent, the value of $\deg(f_1)$ in $\mathbb{R}[x,y]$ must be at least one. It follows that $\deg(q)=2\,$, $\deg(f_1)=1\,$, and $\deg(p)=1\,$. Supposing that $f_1$ is of the form $p^2$ or $p^2\pm 1$ leads to contradiction, as $\deg(f_1)$ could not be one. Hence, $f_1$ must be of the form $p^2+q+I$ for some $I \in \mathbb{R}\,$. Observing that the curves of $\mathbb{R}[x,y]$ associated with $x^2+y+I=0$ and $x=0$ have exactly one intersection point, it follows that the lines of $\mathbb{R}[x,y]$ associated with $f_1=0$ and $p=0$ are crossing lines. Hence, we can apply some $\theta \in \Gamma_2(\mathbb{R})$ such that $\theta(f_1)=-y$ and $\theta(p)=x\,$. This implies that $\theta(f_2)=f_2$ and yields \[ -y \, = \, \theta(f_1) \, = \, \theta(p^2+q+I) \, = \, x^2 + \theta(q)+I\, . \] Hence, $\theta(q)$ is equal to $-y-x^2-I$ and it follows that \[ \theta\left( \, \varphi(f) \, \right) \, = \,\theta\left( \, q(p^2+q+I) + f_2 \, \right) \, = \, (-y-x^2-I)(-y) + f_2 \, = \, f \, . \] As such, $\varphi(f)$ and $f$ are equivalent via $\sim\,$, and Theorem \ref{thm1} implies that $g$ and $f$ must be equal.
\end{proof}

In the subsequent proposition, the following definition and lemma will prove useful. This definition is adapted from \cite{Ab-H-S}.

\begin{definition}\label{def_sets}
If $f$ is a polynomial in $\mathbb{R}[x,y]\,$, $K$ is a field, and $g$ is a polynomial in $K[x,y]\,$, let
\begin{align*}
\text{cusp}(f) \, &:= \, \{ \, r \in \mathbb{R} \, \mid \, f-r=0 \text{ has a cusp in } \mathbb{R}^2 \, \}\, , \\
\text{isol}(f) \, &:= \, \{ \, r \in \mathbb{R} \, \mid \, f-r=0 \text{ has an isolated point in } \mathbb{R}^2 \, \}\, , \\
\text{node}(f) \, &:= \, \{ \, r \in \mathbb{R} \, \mid \, f-r=0 \text{ has a node in } \mathbb{R}^2 \, \}\, , \\
\text{red}(g) \, &:= \, \{ \, s \in K \, \mid \, g-s \text{ is reducible in } K[x,y] \, \} \, , \\
\text{sing}(g) \, &:= \, \{ \, s \in K \, \mid \, g-s=0 \text{ has a singular point in } K^2 \, \} \, .
\end{align*}
\end{definition}

\begin{lemma}\label{lemma_sets}
Let $f,g$ be polynomials in $K[x,y]\,$. If $f \approx_K g\,$, then
\begin{align*}
\emph{red}(f) \, &= \, \{ \, cr \, \mid \, r \in \emph{red}(g) \, \} \, .
\end{align*}
A similar statement can be made about $\emph{sing}(f)\,$. Should $K$ equal $\mathbb{R}\,$, additional similar statements can be made about $\emph{cusp}(f)\,$, $\emph{isol}(f)\,$, and $\emph{node}(f)\,$.
\end{lemma}

\begin{proof}
Let $\phi \in \text{Aut} \, K[x,y]$ be such that $\phi$ fixes $K$ and $\phi(f)=cg\,$. Observe that $\phi(f-cr) \, = \, c(g-r)\,$.
\end{proof}

The information in Table \ref{table_sets} is straightforward to compute. The contents will be useful when proving the proposition below.

\begin{table}
\def\arraystretch{1.5}
\begin{tabular}{| >{\centering}m{1in} | >{\centering}m{0.5in} | >{\centering}m{0.5in} | >{\centering}m{0.5in} | >{\centering}m{0.5in} |}
\hline
$f$ & $\text{cusp}(f)$ & $\text{isol}(f)$ & $\text{node}(f)$ & $\text{red}(f)$
\tabularnewline \hline\hline
$x^3-y^2+x+J$ & $\emptyset$ & $\emptyset$ & $\emptyset$ & $\emptyset$
\tabularnewline \hline
$x^3-y^2-x+J$ & $\emptyset$ & $J+\frac{2\sqrt{3}}{9}$ & $J-\frac{2\sqrt{3}}{9}$ & $\emptyset$
\tabularnewline \hline
$x^3-y^2\pm 1$ & $\pm 1$ & $\emptyset$ & $\emptyset$ & $\emptyset$
\tabularnewline \hline
$x^3-y^2$ & $0$ & $\emptyset$ & $\emptyset$ & $\emptyset$
\tabularnewline \hline
$x^3-y$ & $\emptyset$ & $\emptyset$ & $\emptyset$ & $\emptyset$
\tabularnewline \hline
$x^3-xy+1$ & $\emptyset$ & $\emptyset$ & $\emptyset$ & $1$
\tabularnewline \hline
$x^3-xy$ & $\emptyset$ & $\emptyset$ & $\emptyset$ & $0$
\tabularnewline \hline
\end{tabular}
\newline
\caption{For polynomials in Proposition \ref{auto_x^3}}\label{table_sets}
\end{table}

\begin{proposition}\label{auto_x^3}
Assume $f,g$ are polynomials listed in Table \ref{affine} with canonical form $x^3\,$. If $f \approx g\,$, then $f = g\,$.
\end{proposition}

\begin{proof}
Let $\varphi \in \text{Aut}\,\mathbb{R}[x,y]$ be such that $\varphi(f)=cg$ for some $c \in \mathbb{R}^{\times}\,$, and express $\varphi$ in the form $\big\langle \, p \, , \, q \, \big\rangle\,$. If $f$ is a polynomial only in $x\,$, then $\varphi(f)$ is determined completely by $p\,$. Supposing $\deg(p)$ is greater than one leads to contradiction since $\varphi(f)$ would then be of degree at least six. Hence, $\deg(p)$ must be one, and the result from Theorem \ref{thm1} applies. It remains to consider the polynomials that are the subject of Table \ref{table_sets}. By Lemma \ref{lemma_sets}, our desired result follows should $f$ be either $x^3-xy\,$, $x^3-xy+1\,$, or $x^3-y^2\,$. Lemma \ref{lemma_sets} also implies that $f,g$ must both be contained in one of the following sets
\begin{align}
& \{ \, x^3-y^2 \pm 1 \, \} \, , \\
& \{ \, x^3-y^2-x+J \, \} \, , \\
& \{ \, x^3-y \, , \, x^3-y^2+x+J \, \} \, .
\end{align}

Assume $f,g$ are from set (10), and suppose $f$ is not equal to $g\,$. By Lemma \ref{lemma_sets} and the automorphism property of $\varphi\,$, we have that $c$ equals $-1$ and subsequently that $\varphi(x^3-y^2)$ equals $-x^3+y^2\,$. Since $\varphi$ must map the singular points of $f$ to singular points of $g\,$, it follows that $\varphi$ must fix the origin. Hence, the constant coefficient of $p,q$ must be zero. Expressing $q$ in terms of homogeneous components $q_i$ each of degree $i\,$, it follows that ${-q_1}^2$ equals $y^2\,$. This leads to contradiction since ${-1}$ does not have a square root in $\mathbb{R}\,$.

Assume $f,g$ are from set (11), and refer to $f$ by $x^3-y^2-x+J_1$ and to $g$ by $x^3-y^2-x+J_2\,$. By Lemma \ref{lemma_sets}, we have the system
\begin{align}
J_1+\frac{2\sqrt{3}}{9} \, &= \, c\left(J_2 + \frac{2\sqrt{3}}{9}\right) \, , \\
J_1-\frac{2\sqrt{3}}{9} \, &= \, c\left(J_2 - \frac{2\sqrt{3}}{9}\right) \, .
\end{align}
Subtracting (13) from (14) to eliminate $J_1,J_2$ yields that $c$ equals $1$ and our desired result follows.

Assume $f,g$ are from set (12). Viewing these polynomials in $\mathbb{C}[x,y]\,$, it is straightforward to show that $\text{sing}(x^3-y) \, = \, \emptyset$ and that
\[ \text{sing}(x^3-y^2+x+J) \, = \, \left\lbrace \, J-\frac{2i\sqrt{3}}{9} \, , \, J+\frac{2i\sqrt{3}}{9} \, \right\rbrace \, . \] It follows from Lemma \ref{lemma_sets} that our desired result follows for $x^3-y\,$. Henceforth, refer to $f$ by $x^3-y^2+x+J_1$ and to $g$ by $x^3-y^2+x+J_2\,$. Note that $\varphi$ can be uniquely be extended to $\overline{\varphi} \in \text{Aut} \, \mathbb{C}[x,y]$ by defining $\overline{\varphi}(i)=i\,$, and it follows that $f \approx_{\mathbb{C}} g\,$. Hence, Lemma \ref{lemma_sets} applies and two possible systems of equations result. One such system is similar to (13) and (14), resulting in the solution $c=1$ and our desired result follows. The other such system is
\begin{align}
J_1+\frac{2i\sqrt{3}}{9} \, &= \, c \left(J_2-\frac{2i\sqrt{3}}{9} \right) \, , \\
J_1-\frac{2i\sqrt{3}}{9} \, &= \, c \left(J_2+\frac{2i\sqrt{3}}{9} \right) \, . 
\end{align}
Suppose this system holds. Subtracting (15) from (16) to eliminate $J_1,J_2\,$ yields that $c$ equals $-1$ (ie. $\overline{\varphi}(f)=-g\,$) and $J_2$ equals $-J_1\,$. Let $\theta$ be in $\text{Aut} \, \mathbb{C}[x,y]$ such that it fixes $\mathbb{C}$ and is of the form $\big\langle x+ \frac{i\sqrt{3}}{3} \, , \, y \big\rangle\,$. It follows that \[ \theta^{-1}\bigg( \, f-J_1+\frac{2i\sqrt{3}}{9} \, \bigg) \, = \, x^3-i\sqrt{3}\,x^2-y^2\, .\]
Combining this with $\overline{\varphi}(f)=-g$ yields
\begin{align}
\left( \, \theta \circ \overline{\varphi} \circ \theta \, \right)\left( \, x^3-i\sqrt{3}\,x^2-y^2 \, \right) \, = \, -x^3-i\sqrt{3}\,x^2+y^2\, .
\end{align}
Express $\theta \circ \overline{\varphi} \circ \theta$ in the form $\big\langle \, P \, , \, Q \, \big\rangle\,$, and let $P_i,Q_i$ be the homogeneous component of $P,Q$ of degree $i\,$. Observe that all of the coefficients of $Q$ must be real. Also observe that $\theta \circ \overline{\varphi} \circ \theta$ must fix the origin as it sends the singular point of $x^3-(i\sqrt{3})x^2-y^2$ to the singular point of $-x^3+(i\sqrt{3})x^2-y^2$. Hence, $P_0,Q_0$ must both be zero. As such, isolating the degree-two terms of (17) results in
\begin{align}
-\left( \, i\sqrt{3}\,{P_1}^2+{Q_1}^2 \, \right) \, = \, -\left( \, i\sqrt{3}\, x^2-y^2\, \right) \, .
\end{align}
Express $P_1$ as $(a_1+a_2i)x+(b_1+b_2i)y$ and $Q_1$ as $cx+dy\,$, with $a_1,a_2,b_1,b_2,c,d \in \mathbb{R}\,$. Inspecting the $x^2$-, $xy$-, and $y^2$-coefficients of (18) yields the system
\begin{align}
i\sqrt{3}\,(a_1+a_2i)^2+c^2 \, &= \, i\sqrt{3} \, , \\
i\sqrt{3}\,(a_1+a_2i)(b_1+b_2i)+cd \, &= \, 0 \, , \\
i\sqrt{3}\,(b_1+b_2i)^2+d^2 \, &= \, {-1} \, .
\end{align}
From (21), it follows that ${b_1}^2$ must equal ${b_2}^2$ (and both cannot be zero). From (20), it follows that $a_1 b_2$ must equal $a_2 b_1$ (and hence ${a_1}^2$ and ${a_2}^2$ are equal). This leads to contradiction, as (19) demands that ${a_1}^2$ and ${a_2}^2$ must differ by $1\,$. As a result, the system of (15) and (16) is logically impossible, and we have achieved our desired result.
\end{proof}

We now have established an automorphic classification of cubic curves, which incidentally coincides with our affine classification. That said, it is worth noting that some polynomials $f$ in our affine classification were such that $\text{AutDeg}(f)$ was less than three. Such polynomials are listed in Table \ref{autdeg} along with a representative of minimal degree.

\begin{theorem}\label{thm2}
Every degree-three polynomial in $\mathbb{R}[x,y]$ is equivalent with respect to $\approx$ to some polynomial in Table \ref{affine}. Moreover, the polynomials in Table \ref{affine} are pairwise inequivalent with respect to $\approx\,$.
\end{theorem}

\begin{proof}
The result follows from the propositions in this section.
\end{proof}

\begin{table}
\def\arraystretch{1.5}
\begin{tabular}{| >{\centering}m{1in} | >{\centering}m{1in} | >{\centering}m{1in} |}
\hline
$f$ & $\varphi$ & $\varphi(f)$
\tabularnewline \hline\hline
$x^3-y$ & $\big\langle \, y \, , \, y^3-x \, \big\rangle$ & $x$
\tabularnewline \hline
$x^3-xy$ & $\big\langle \, y \, , \, y^2-x \, \big\rangle$ & $xy$
\tabularnewline \hline
$x^3-xy+1$ & $\big\langle \, y \, , \, y^2-x \, \big\rangle$ & $xy+1$
\tabularnewline \hline
\end{tabular}
\newline
\caption{Cubic polynomials with $\text{AutDeg}(f) < 3$}\label{autdeg}
\end{table}

As a concluding result, we will note that the polynomials in Proposition \ref{auto_x^3} of the form $x^3-y^2+\lambda(f)$ with $\lambda(f)$ being linear could be alternatively handled (more concisely) using the Epimorphism Theorem of Abhyankar and Moh \cite{Ab-Moh}. The subsequent proposition demonstrates this approach.

\begin{proposition}
Let $f,g$ be polynomials from Table \ref{affine} of the form $x^3-y^2+\lambda(f)$ and $x^3-y^2+\lambda(g)$ where $\lambda(f),\lambda(g)$ are linear. If $f \approx g\,$, then $f =g\,$.
\end{proposition}

\begin{proof}
Suppose $\varphi$ is in $\text{Aut} \, \mathbb{R}[x,y] \setminus \Gamma_2(\mathbb{R})$ such that $\varphi(f)=cg$ for some $c \in \mathbb{R}^{\times}\,$. Express $\varphi$ in the form $\big\langle \, p \, , \, q \, \big\rangle$ and observe that $\deg(p),\deg(q)$ must equal $2k,3k$ (repectively) for some positive integer $k\,$. Apply $\theta_a$ of the form $\big\langle \, x+ay \, , \, y \, \big\rangle$ for some $a \in \mathbb{R}$ so that $\theta_a(p)$ will have a nonzero $y^{2k}$-coefficient and $\theta_a(q)$ will have a nonzero $y^{3k}$-coefficient. Such an $a$ is possible since the $y^{2k}$- and $y^{3k}$-coefficients of $p$ and $q$ (respectively) will be polynomials in $\mathbb{R}[a]$ and hence have only finitely many roots in $\mathbb{R}\,$. View $\overline{\theta \circ \varphi}$ as the natural extension of $\theta \circ \varphi$ in $\text{Aut} \, \mathbb{C}[x,y]$ that fixes $\mathbb{C}\,$, and express $\overline{\theta \circ \varphi}$ as $\big\langle \, P \, , \, Q \, \big\rangle\,$. It follows that the degree of $P(0,y),Q(0,y)$ equals $2k,3k$ (respectively). This leads to contradiction of Abhyankar-Moh \cite{Ab-Moh} as neither $2k$ nor $3k$ will divide the other.  
\end{proof}

This work is adapted, in part, from the author's Ph.D. thesis \cite{Bly}.

\end{document}